\documentclass[12pt]{article}

\usepackage{amsmath, amsthm, amssymb, fullpage, comment}
\usepackage{graphicx}
\usepackage{tikz}
\usepackage{dsfont} 

\usepackage{graphics}
\usepackage{graphicx}   

\usepackage{float}

\usepackage{booktabs}

\newtheorem{theo}{Theorem}[section]
\newtheorem{prop}[theo]{Proposition}
\newtheorem{lemma}[theo]{Lemma}
\newtheorem{coro}[theo]{Corollary}

\usetikzlibrary{shapes.geometric}
\usetikzlibrary{decorations.markings}

\pgfdeclarelayer{edgelayer}
\pgfdeclarelayer{nodelayer}
\pgfsetlayers{edgelayer,nodelayer,main}

\tikzstyle{none}=[inner sep=0pt]
\tikzstyle{none}=[inner sep=0pt]
 \definecolor{hexcolor0xff0000}{rgb}{1.000,0.000,0.000}
\definecolor{hexcolor0x000000}{rgb}{0.000,0.000,0.000}
 \definecolor{hexcolor0x00ff00}{rgb}{0.000,1.000,0.000}
 \definecolor{hexcolor0xffff00}{rgb}{1.000,1.000,0.000}
 \definecolor{hexcolor0x000000}{rgb}{0.000,0.000,0.000}
 \definecolor{hexcolor0xffffff}{rgb}{1.000,1.000,1.000}
 \definecolor{hexcolor0x000000}{rgb}{0.000,0.000,0.000}
 \definecolor{hexcolor0x000000}{rgb}{0.000,0.000,0.000}
 \definecolor{hexcolor0x000000}{rgb}{0.000,0.000,0.000}
 \definecolor{hexcolor0xff0000}{rgb}{1.000,0.000,0.000}
 \definecolor{hexcolor0xff0000}{rgb}{1.000,0.000,0.000}
 \definecolor{hexcolor0x000000}{rgb}{0.000,0.000,0.000}
 \definecolor{hexcolor0x00ff00}{rgb}{0.000,1.000,0.000}
 \definecolor{hexcolor0xffff00}{rgb}{1.000,1.000,0.000}
 \definecolor{hexcolor0x000000}{rgb}{0.000,0.000,0.000}
 \definecolor{hexcolor0xffffff}{rgb}{1.000,1.000,1.000}
 \definecolor{hexcolor0x000000}{rgb}{0.000,0.000,0.000}
 \definecolor{hexcolor0xffffff}{rgb}{1.000,1.000,1.000}
 \definecolor{hexcolor0xffffff}{rgb}{1.000,1.000,1.000}
 \definecolor{hexcolor0xffffff}{rgb}{1.000,1.000,1.000}
 \definecolor{hexcolor0x000000}{rgb}{0.000,0.000,0.000}
 \definecolor{hexcolor0x0200ff}{rgb}{0.008,0.000,1.000}
 \definecolor{hexcolor0x1cff00}{rgb}{0.110,1.000,0.000}
 \definecolor{hexcolor0x000000}{rgb}{0.000,0.000,0.000}
 \definecolor{hexcolor0xfff300}{rgb}{1.000,0.953,0.000}
 \definecolor{hexcolor0x000000}{rgb}{0.000,0.000,0.000}
 \definecolor{hexcolor0x000000}{rgb}{0.000,0.000,0.000}
 \definecolor{hexcolor0x000000}{rgb}{0.000,0.000,0.000}
 \definecolor{hexcolor0xff0000}{rgb}{1.000,0.000,0.000}
 \definecolor{hexcolor0x000000}{rgb}{0.000,0.000,0.000}

\tikzstyle{SmallVert}=[circle,fill=hexcolor0xffffff,draw=hexcolor0x000000, scale = 0.65]
\tikzstyle{arrow}=[-,draw=hexcolor0x000000,postaction={decorate},decoration={markings,mark=at position .5 with {\arrow{>}}},line width=1.000]
\tikzstyle{simple}=[-,draw=hexcolor0x000000,line width=1.000]
\tikzstyle{RedEdge}=[-,draw=hexcolor0x000000,line width=2.600]
\tikzstyle{dashedred}=[-,draw=hexcolor0x000000,line width=2.600, dashed]
\tikzstyle{dashedblack}=[-,draw=hexcolor0x000000,line width=1.000, dashed]
\tikzstyle{Clearer}=[circle,fill=hexcolor0xffffff,draw=hexcolor0xffffff]
\tikzstyle{Arrow}=[->,draw=hexcolor0x000000,line width=2.400]
\tikzstyle{redvert}=[rectangle,fill=hexcolor0xffffff,draw=hexcolor0x000000,scale=0.65]
\tikzstyle{bluevert}=[regular polygon,regular polygon sides=3,shape border rotate=0,fill=hexcolor0xffffff,draw=hexcolor0x000000,scale=0.7]
\tikzstyle{greenvert}=[shape=diamond,fill=hexcolor0xffffff,draw=hexcolor0x000000,scale=0.85]
\tikzstyle{yellowvert}=[circle,fill=hexcolor0xffffff,draw=hexcolor0x000000,scale=0.75]
\tikzstyle{tinyvert}=[circle,fill=hexcolor0xffffff,draw=hexcolor0x000000,scale=0.1]
\tikzstyle{point}=[rectangle,fill=hexcolor0x000000,draw=hexcolor0x000000,scale=0.08]

\begin{document}

\title{The complexity of frugal digraph homomorphisms}
	\author{Stefan Bard\thanks{Department of Mathematics and Statistics, University of Victoria, BC, Canada}, 
	Gary MacGillivray\thanks{Department of Mathematics and Statistics, University of Victoria, BC, Canada;\ Research supported by NSERC}, 
	Jacobus Swarts\thanks{Department of Mathematics, Vancouver Island University, Nanaimo, BC, Canada}}
	\date{}
	\maketitle

\begin{abstract}
    For an integer $t \geq 1$, a homomorphism of a digraph G to a digraph $H$ is {\it $t$-frugal} if no more than $t$ in-neighbours of any vertex of $G$ have the same image.  There is a dichotomy theorem based on structural properties when $t=1$ and $H$ has a loop at each vertex, but there is unlikely to be such a theorem for general digraphs $H$.  For $t \geq 2$ we describe a structural dichotomy theorem for $t$-frugal homomorphisms of general digraphs.
\end{abstract}


\section{Introduction}

A homomorphism of a digraph $G$ to a digraph $H$ is a function $f: V(G) \to V(H)$ such that if $xy$ is an arc of $G$, then $f(x)f(y)$ is an arc of $H$.    Note that if $H^\prime$ is obtained from $H$  by deleting all but one member of each collection of arcs with the same ends and orientation and $G^\prime$ is obtained from $G$ in the same way, then a digraph $G$ has a homomorphism to $H$ if and only if $G$ has a homomorphism to $H^\prime$ if and only if $G^\prime$ has a homomorphism to $H^\prime$.  Thus we assume  that no two vertices of any digraph are joined by more than one arc with the same ends and orientation.

Homomorphisms generalize colourings.  
For example, a proper $k$-colouring of a graph $G$ can be defined as, and is equivalent to, a homomorphism of $G$ to some simple graph on $k$ vertices.  From this perspective one can see that the dichotomy theorem stating the $k$-colouring problem is solvable in polynomial time if $k \leq 2$ and NP-complete if $k \geq 3$  generalizes to the theorem that the problem of deciding whether there is a homomorphism of a given  graph $G$ to a fixed graph $H$ is solvable in polynomial time if $H$ is bipartite or has a loop, and is NP-complete if $H$ is a non-bipartite simple graph \cite{HN}.

There is (also) a dichotomy theorem for the complexity of deciding whether there is a homomorphism of a given digraph $G$ to a fixed digraph $H$;  the dichotomy has an algebraic description \cite{B17, Z17}. 
Because every constraint satisfaction problem is polynomially equivalent to a problem involving homomorphisms of digraphs \cite{FederVardi}, it seems likely that a structural description would be very complicated.

For an integer $t \geq 1$, a homomorphism of a digraph $G$ to a digraph $H$ is {\it $t$-frugal} if no more than $t$ in-neighbours of any vertex of $G$ have the same image. The 1-frugal homomorphisms, also known as {\it locally-injective homomorphisms}, were studied in \cite{inj8}. In this paper, the authors prove the following results.
\begin{itemize}
    \item A structural dichotomy theorem for the problem of deciding the existence of a 1-frugal homomorphism to reflexive digraphs (digraphs with a loop at each vertex). The dichotomy states that if $H$ is a reflexive digraph, which is also a core, the problem of deciding the existence of a $1$-frugal homomorphism to $H$ is solvable in polynomial time when $H$ is a loop, a reflexive arc or a reflexive $2$-cycle, and NP-complete otherwise.
    \item For irreflexive (no loops) digraphs $H$, the problem of deciding the existence of a 1-frugal homomorphism to $H$ was shown to be NP-complete when the maximum in-degree of $H$ is at least $3$, and solvable in polynomial time when $H$ has maximum in-degree 1.
    \item It was further shown that any directed graph homomorphism is polynomially equivalent to a 1-frugal homomorphism  problem in which $H$ has maximum in-degree $2$.
\end{itemize} 
Thus, for the same reason as above, it seems likely that a structural description of the dichotomy for deciding the existence of a 1-frugal homomorphism to an irreflexive digraph would be very complicated; hence the same comment applies to general digraphs $H$.

In this paper we consider the complexity of $t$-frugal homomorphisms for $t \geq 2$.  In contrast to the $t = 1$ case, we are able to establish a dichotomy theorem for which the dichotomy admits a structural description.  Further, the description makes no additional assumptions about $H$ --- it could have directed 2-cycles or loops at some vertices.

Other definitions of frugality are possible for digraphs since each vertex of a digraph has two neighbourhoods: an in-neighbourhood and an out-neighbourhood.  The condition that no more than $t$ in-neighbours of any vertex are assigned the same image can be (equivalently) applied to out-neighbourhoods, to in- and out-neighbourhoods separately, or to in- and out-neighbourhoods together.  In addition, the target digraph $H$ can be restricted to be reflexive or irreflexive.  Problems of this type have been considered  \cite{BBDMY, Russell, Rus1}.

There is a rich history of 1-frugal homomorphisms of graphs.  Further, each homomorphism problem gives rise to an associated colouring problem by defining the `chromatic number' to be the minimum number of vertices in a graph or digraph to which there exists a homomorphism of the desired type.  For more information see  \cite{SBard} and the references therein.

\section{Special versions of SAT}

In this section we establish the complexity of several special versions of the satisfiability problem that will be used in the proof of the main result.  The problem \emph{monotone $\ell$-in-$k$-SAT} is described as follows: for $1 \leq \ell \leq k$, given Boolean variables $x_1, x_2, \ldots, x_n$ and a collection of $k$-variable clauses over them in which no clause contains a negated variable, is there a truth assignment to the variables such that exactly $\ell$ of the variables in each clause are true?

\begin{prop}
\label{SATswitch}
For $1 \leq \ell < k$, a monotone collection of $k$-variable clauses has a satisfying truth assignment where each clause has exactly $\ell$ true variables if and only if it has a satisfying truth assignment where each clause has exactly $k - \ell$ true variables.
\end{prop}

\begin{proof}
Simply switch the truth value of every variable in a valid monotone $\ell$-in-$k$ truth assignment to get a valid monotone $(k-\ell)$-in-$k$ truth assignment.	
\end{proof}

As such, we consider only the complexity of monotone $\ell$-in-$k$-SAT for $\ell \leq k/2$.

\begin{prop}
\label{P2}
If $\ell \leq k/2$, then monotone $\ell$-in-$k$-SAT polynomially transforms to monotone $\ell$-in-$(k+1)$-SAT.
\end{prop}

\begin{proof}
	
Consider an instance $S$ of monotone $\ell$-in-$k$-SAT, where $\ell \leq k/2$. We construct an instance $S'$ of monotone $\ell$-in-$(k+1)$-SAT from $S$ as follows. Introduce $\ell + 1$ new variables $x_1, x_2, \ldots, x_{\ell + 1}$. For each clause $c$ in $S$, obtain $\ell + 1$ new clauses by adding each of the variables $x_i$, $1 \leq i \leq \ell + 1$, to $c$. Next, introduce $k - \ell$ new variables $y_1,y_2,\ldots, y_{k - \ell}$. Add a new clause consisting of $x_1, x_2,\ldots, x_{\ell + 1}, y_1,y_2,\ldots, y_{k - \ell}$. This completes the construction of $S'$. This construction can clearly be carried out in polynomial time.

Consider a satisfying truth assignment of $S'$ where each clause has exactly $\ell$ true variables. Notice that $x_1, x_2, \ldots, x_{\ell + 1}$ must be assigned the same truth value. If $x_i$ and $x_j$ are assigned different truth values, where $1 \leq i < j \leq \ell +1$, then for a clause $c$ in $S$, the corresponding clauses in $S'$ obtained by appending $x_i$ and $x_j$ to $c$ cannot each have $\ell$ true variables. Next, notice that a satisfying assignment where each clause has exactly $\ell$ true variables must assign $x_i$ the value false for each $i$, as otherwise the clause containing $x_1,x_2,\ldots,x_{\ell+1}$ will not have exactly $\ell$ true variables. Therefore, we can obtain a satisfying truth assignment of $S$  where each clause has exactly $\ell$ true variables by allowing the truth value of each variable to be equal to the truth value of the corresponding variable in our satisfying truth assignment of $S'$ where each clause has exactly $\ell$ true variables.

Now, consider a satisfying truth assignment of $S$ where each clause has exactly $\ell$ true variables. To obtain a satisfying truth assignment of $S'$ where each clause has exactly $\ell$ true variables, we first let the truth values of each variable in $S'$ which corresponds to a variable in $S$ be the truth value of that variable in our satisfying truth assignment of $S$ where each clause has exactly $\ell$ true variables. Then, assign $x_i$, $1 \leq i \leq \ell+1$, the value false. Then, assign exactly $\ell$ variables among $y_1,y_2,\ldots, y_{k - \ell}$ the value true, which can be done because $\ell \leq k/2$, and therefore $k - \ell \geq 2\ell - \ell = \ell$. Clearly, the clause consisting of $x_1, x_2,\ldots, x_{\ell + 1}, y_1,y_2,\ldots, y_{k - \ell}$ will have exactly $\ell$ true variables with this assignment. Furthermore, each other clause corresponds to a clause of $S$, with an additional false variable $x_i$, and because we started with a satisfying truth assignment of $S$, these clauses must each have exactly $\ell$ true variables as well. Therefore, each clause of $S'$ has exactly $\ell$ true variables in our constructed assignment, and therefore we have a satisfying truth assignment of $S$ where each clause has exactly $\ell$ true variables. This completes the proof.
	
\end{proof}

\begin{prop}
\label{P3}
Monotone $(\ell-1)$-in-$(2\ell-1)$-SAT polynomially transforms to monotone $\ell$-in-$(2\ell)$-SAT.
\end{prop}

\begin{proof}
	
Consider an instance $S$ of monotone $(\ell-1)$-in-$(2\ell-1)$-SAT. We construct an instance $S'$ of monotone $\ell$-in-$(2\ell)$-SAT from $S$ as follows. First, introduce a new variable $x$. For each clause $c$ in $S$, add to $S'$ the clause obtained by appending $x$ to $c$. This completes the construction of $S'$.

We must show that there is a satisfying truth assignment of $S'$ where each clause has exactly $\ell$ true variables if and only if there is a satisfying truth assignment of $S$ where each clause has exactly $\ell - 1$ true variables. Consider first a satisfying truth assignment of $S'$ where each clause has exactly $\ell$ true variables. There are two cases depending on the truth value of $x$.

Suppose first, we have a solution of $S'$ where $x$ is true. In this case, the assignment induced by this solution of $S'$ is in fact a solution of $S$, as each clause of $S$ has exactly one fewer true variable than the clauses of $S'$ under this assignment. Suppose next that we have a solution of $S'$ where $x$ is false. In this case, by construction, in each clause of $S$, exactly $\ell$ of the $2\ell -1$ variables are true. By Proposition~\ref{SATswitch}, we may simply switch the truth value of each variable in this assignment to receive an assignment where exactly $\ell - 1$ of the $2\ell -1$ variables in each clause are true, as desired.

Now, consider a satisfying truth assignment of $S$ where each clause has exactly $\ell - 1$ true variables. Clearly, a satisfying truth assignment of $S'$  where each clause has exactly $\ell$ true variables can be obtained from this by simply assigning $x$ the value true. This completes the proof.
\end{proof}
The previous propositions together give the following:

\begin{coro}
\label{linkSAT}
Let $k \geq 3$ be an integer. Then, monotone $\ell$-in-$k$-SAT is NP-complete.
\end{coro}

\begin{proof}
It is known that monotone $1$-in-$3$-SAT (and hence monotone $2$-in-$3$-SAT) is NP-complete \cite{Schaefer}. Suppose that $\ell$-in-$k$-SAT is NP-complete for some $k \geq 3$ and all $\ell$, $1 \leq \ell \leq k-1$. Consider monotone $\ell'$-in-$(k+1)$-SAT. (By Proposition~\ref{SATswitch}, we need only consider values of $\ell'$ such that $1 \leq \ell' \leq (k+1)/2$.) If $\ell' \leq k/2$, then by Proposition~\ref{P2} monotone $\ell'$-in-$k$-SAT polynomially transforms to monotone $\ell'$-in-$(k+1)$-SAT. If $\ell' > k/2$, then $\ell' = (k+1)/2$. Thus, by Proposition~\ref{P3}, monotone $(\ell'-1)$-in-$k$-SAT polynomially transforms to monotone $\ell'$-in-$(k+1)$-SAT. In either case, the proof is complete by induction.
\end{proof}

\section{The main construction}

Let $G$ and $H$ be a directed graphs. We use the term \emph{$H$-colouring of $G$ } to refer to a homomorphism  of $G$ to $H$. In this section we determine, for each digraph $H$ (which may contain loops but not multiple arcs) and for all $t \geq 2$, the complexity of the following decision problem.



\begin{center}
\begin{tabular}{p{0.8\textwidth}}\toprule
\textbf{Problem:} $t$-frugal $H$-colouring\\\midrule
\hspace{1.5cm} \textbf{Instance:} An irreflexive  directed graph $G$.\\

\hspace{1.5cm} \textbf{Question:} Does there exist a $t$-frugal $H$-colouring of $G$?\\\bottomrule
\end{tabular}
\end{center}

Given a directed graph $H$, we define the graph $H^*$ as the graph with vertex set $V(H)$ and with $uv \in E(H^*)$ if and only if (in $H$) $u$ and $v$ are distinct, $u$ and $v$ have a common out-neighbour, and each of $u$ and $v$ has an out-neighbour in $H$ with in-degree $\Delta^-(H)$. This construction is inspired by the indicator construction in 
\cite{HN}.

\begin{theo}
\label{reduc}
Let $H$ be a directed graph, and let $t \geq 2$. Then, $H^*$-colouring polynomially transforms to $t$-frugal $H$-colouring.
\end{theo}

\begin{proof}
Let $G$ be an instance of $H^*$-colouring. We will give a construction to transform the graph $G$ into a directed graph $^*G$, and we will show that $^*G$ has a $t$-frugal $H$-colouring if and only if $G$ has an $H^*$ colouring. We first construct some gadgets that will be useful in the construction, and determine some of their properties.

First, let $w$, $x$, $y$, and $z$ be vertices. Add an arc from $w$ to $x$ and from $z$ to $y$. Let there be $t\Delta^-(H) - 1$ vertices each with $x$ and $y$ as out-neighbours. Call the resulting graph $F_0$ (See Figure~\ref{F0.}).

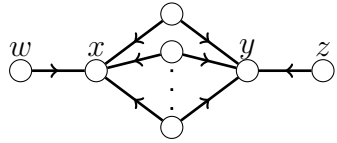
\begin{figure}[ht]
\begin{center}
	\begin{tikzpicture}
	\begin{pgfonlayer}{nodelayer}
	\node [style=yellowvert] (0) at (-2, -0) {};
	\node [style=yellowvert] (1) at (-1, -0) {};
	\node [style=yellowvert] (2) at (1, -0) {};
	\node [style=yellowvert] (3) at (2, -0) {};
	\node [style=yellowvert] (4) at (0, 0.75) {};
	\node [style=yellowvert] (5) at (0, 0.25) {};
	\node [style=yellowvert] (6) at (0, -0.75) {};
	\node [style=point] (7) at (0, -0) {};
	\node [style=point] (8) at (0, -0.25) {};
	\node [style=point] (9) at (0, -0.5) {};
	\node [] (10) at (-2, 0.285) {$w$};
	\node [] (11) at (-1, 0.285) {$x$};
	\node [] (12) at (1, 0.285) {$y$};
	\node [] (13) at (2, 0.285) {$z$};
	\end{pgfonlayer}
	\begin{pgfonlayer}{edgelayer}
	\draw [style=arrow] (0) to (1);
	\draw [style=arrow] (6) to (1);
	\draw [style=arrow] (5) to (1);
	\draw [style=arrow] (4) to (1);
	\draw [style=arrow] (4) to (2);
	\draw [style=arrow] (5) to (2);
	\draw [style=arrow] (6) to (2);
	\draw [style=arrow] (3) to (2);
	\end{pgfonlayer}
	\end{tikzpicture}
	\caption{The digraph $F_0$.}
	\label{F0.}
	\end{center}
\end{figure}

We claim that in any $t$-frugal $H$-colouring of $F_0$, $w$ and $z$ must receive the same colour. To see this, notice that the among the $t\Delta^-(H) - 1$ vertices which are common in-neighbours of $x$ and $y$, at most $\Delta^-(H)$ different colours can appear, and each such colour can only appear at most $t$ times. The only way for these $t\Delta^-(H) - 1$ vertices to be coloured meeting these criteria is for $\Delta^-(H) - 1$ colours to colour exactly $t\Delta^-(H) - t$ of the vertices, and for one colour, call it $a$, to colour exactly $t-1$ of these vertices. In order to satisfy the frugality condition, each of $w$ and $z$ must receive the colour $a$, which proves the claim.

Next, let $F_0^1, F_0^2, \ldots, F_0^{t-1}$ be copies of $F_0$, and denote the vertices of $F_0^i$ corresponding to the vertices $w$, $x$, $y$, and $z$ of $F_0$ as $w^i$, $x^i$, $y^i$, and $z^i$, respectively. Now, we assemble our copies of $F_0$ into a chain by identifying the vertices $z^i$ and $w^{i+1}$ for $1 \leq i \leq t-2$, and call the resulting vertex $u_{i+1}$. Rename $w^1$ to $u_1$, and rename $z^{t-1}$ to $u_{t-1}$. Call the resulting graph $F_1$ (See Figure~\ref{F1.}). Then, by the previous claim, we know that in any $t$-frugal $H$-colouring of $F_1$, the vertices $u_1, u_2,\ldots, u_{t-1}$ must all receive the same colour.

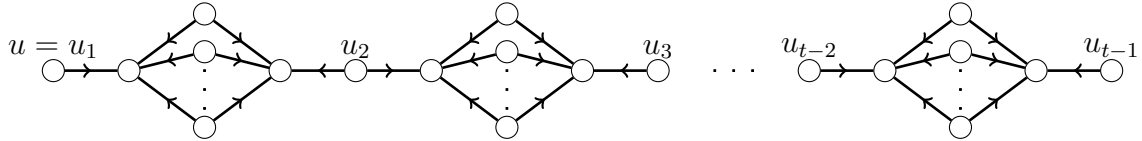
\begin{figure}[ht]
\begin{center}
	\begin{tikzpicture}
	\begin{pgfonlayer}{nodelayer}
	\node [style=yellowvert] (0) at (-8, -0) {};
	\node [style=yellowvert] (1) at (-7, -0) {};
	\node [style=yellowvert] (2) at (-5, -0) {};
	\node [style=yellowvert] (3) at (-4, -0) {};
	\node [style=yellowvert] (4) at (-6, 0.75) {};
	\node [style=yellowvert] (5) at (-6, 0.25) {};
	\node [style=yellowvert] (6) at (-6, -0.75) {};
	\node [style=point] (7) at (-6, -0) {};
	\node [style=point] (8) at (-6, -0.25) {};
	\node [style=point] (9) at (-6, -0.5) {};
	\node [] (10) at (-8, 0.285) {$u=u_1$};
	\node [] (11) at (-4, 0.285) {$u_2$};
	\node [style=yellowvert] (12) at (-2, -0.75) {};
	\node [style=yellowvert] (13) at (-2, 0.75) {};
	\node [style=point] (14) at (-2, -0.25) {};
	\node [style=yellowvert] (15) at (-1, -0) {};
	\node [style=point] (16) at (-2, -0.5) {};
	\node [style=point] (17) at (-2, -0) {};
	\node [style=yellowvert] (18) at (-2, 0.25) {};
	\node [style=yellowvert] (19) at (0, -0) {};
	\node [style=yellowvert] (20) at (-3, -0) {};
	\node [] (21) at (0, 0.285) {$u_3$};
	\node [style=point] (22) at (0.75, -0) {};
	\node [style=point] (23) at (1, -0) {};
	\node [style=point] (24) at (1.25, -0) {};
	\node [style=yellowvert] (25) at (4, -0.75) {};
	\node [style=point] (26) at (4, -0.25) {};
	\node [style=yellowvert] (27) at (4, 0.75) {};
	\node [style=yellowvert] (28) at (5, -0) {};
	\node [style=yellowvert] (29) at (2, -0) {};
	\node [style=point] (30) at (4, -0.5) {};
	\node [style=point] (31) at (4, -0) {};
	\node [style=yellowvert] (32) at (6, -0) {};
	\node [style=yellowvert] (33) at (4, 0.25) {};
	\node [] (34) at (2, 0.285) {$u_{t-2}$};
	\node [style=yellowvert] (35) at (3, -0) {};
	\node [] (36) at (6, 0.285) {$u_{t-1}$};
	\end{pgfonlayer}
	\begin{pgfonlayer}{edgelayer}
	\draw [style=arrow] (0) to (1);
	\draw [style=arrow] (6) to (1);
	\draw [style=arrow] (5) to (1);
	\draw [style=arrow] (4) to (1);
	\draw [style=arrow] (4) to (2);
	\draw [style=arrow] (5) to (2);
	\draw [style=arrow] (6) to (2);
	\draw [style=arrow] (3) to (2);
	\draw [style=arrow] (12) to (20);
	\draw [style=arrow] (18) to (20);
	\draw [style=arrow] (13) to (20);
	\draw [style=arrow] (13) to (15);
	\draw [style=arrow] (18) to (15);
	\draw [style=arrow] (12) to (15);
	\draw [style=arrow] (19) to (15);
	\draw [style=arrow] (3) to (20);
	\draw [style=arrow] (29) to (35);
	\draw [style=arrow] (25) to (35);
	\draw [style=arrow] (33) to (35);
	\draw [style=arrow] (27) to (35);
	\draw [style=arrow] (27) to (28);
	\draw [style=arrow] (33) to (28);
	\draw [style=arrow] (25) to (28);
	\draw [style=arrow] (32) to (28);
	\end{pgfonlayer}
	\end{tikzpicture}
	\caption{The digraph $F_1$.}
	\label{F1.}
\end{center}
\end{figure}

Now, take another copy of $F_0$, and let $v_2$ and $v_1$ denote the vertices corresponding to $w$ and $z$ respectively. Take this copy of $F_0$ along with $F_1$ and a new vertex $q$, then add an arc from each of $u_1,u_2,\ldots,u_{t-1},v_2$ and $v_1$ to $q$. Call the resulting graph $F$ (See Figure~\ref{F.}). The graph $F$ will serve as an edge replacement gadget in our transformation of $G$ to $^*G$.

\begin{figure}[htb]
\begin{center}
	\begin{tikzpicture}
	\begin{pgfonlayer}{nodelayer}
	\node [style=yellowvert] (0) at (-6.75, -0) {};
	\node [style=yellowvert] (1) at (-6, -0) {};
	\node [style=yellowvert] (2) at (-4.5, -0) {};
	\node [style=yellowvert] (3) at (-3.75, -0) {};
	\node [style=yellowvert] (4) at (-5.25, 0.75) {};
	\node [style=yellowvert] (5) at (-5.25, 0.25) {};
	\node [style=yellowvert] (6) at (-5.25, -0.75) {};
	\node [style=point] (7) at (-5.25, -0) {};
	\node [style=point] (8) at (-5.25, -0.25) {};
	\node [style=point] (9) at (-5.25, -0.5) {};
	\node [] (10) at (-6.75, -0.4) {$u_1$};
	\node [] (11) at (-3.75, -0.4) {$u_2$};
	\node [style=yellowvert] (12) at (-2.25, -0.75) {};
	\node [style=yellowvert] (13) at (-2.25, 0.75) {};
	\node [style=point] (14) at (-2.25, -0.25) {};
	\node [style=yellowvert] (15) at (-1.5, -0) {};
	\node [style=point] (16) at (-2.25, -0.5) {};
	\node [style=point] (17) at (-2.25, -0) {};
	\node [style=yellowvert] (18) at (-2.25, 0.25) {};
	\node [style=yellowvert] (19) at (-0.75, -0) {};
	\node [style=yellowvert] (20) at (-3, -0) {};
	\node [] (21) at (-0.75, -0.4) {$u_3$};
	\node [style=point] (22) at (-0.25, -0) {};
	\node [style=point] (23) at (0, -0) {};
	\node [style=point] (24) at (0.25, -0) {};
	\node [style=yellowvert] (25) at (2.25, -0.75) {};
	\node [style=point] (26) at (2.25, -0.25) {};
	\node [style=yellowvert] (27) at (2.25, 0.75) {};
	\node [style=yellowvert] (28) at (3, -0) {};
	\node [style=yellowvert] (29) at (0.75, -0) {};
	\node [style=point] (30) at (2.25, -0.5) {};
	\node [style=point] (31) at (2.25, -0) {};
	\node [style=yellowvert] (32) at (3.75, -0) {};
	\node [style=yellowvert] (33) at (2.25, 0.25) {};
	\node [] (34) at (0.75, -0.4) {$u_{t-2}$};
	\node [style=yellowvert] (35) at (1.5, -0) {};
	\node [] (36) at (3.75, -0.4) {$u_{t-1}$};
	\node [style=yellowvert] (37) at (6, -0.75) {};
	\node [style=point] (38) at (6, -0.25) {};
	\node [style=yellowvert] (39) at (6, 0.75) {};
	\node [style=yellowvert] (40) at (6.75, -0) {};
	\node [style=yellowvert] (41) at (4.5, -0) {};
	\node [style=point] (42) at (6, -0.5) {};
	\node [style=point] (43) at (6, -0) {};
	\node [style=yellowvert] (44) at (7.5, -0) {};
	\node [style=yellowvert] (45) at (6, 0.25) {};
	\node [] (46) at (4.5, -0.4) {$v_2$};
	\node [style=yellowvert] (47) at (5.25, -0) {};
	\node [] (48) at (7.5, -0.4) {$v_1$};
	\node [style=yellowvert] (49) at (0, 3) {};
	\node [] (50) at (0, 3.4) {$q$};
	\end{pgfonlayer}
	\begin{pgfonlayer}{edgelayer}
	\draw [style=arrow] (0) to (1);
	\draw [style=arrow] (6) to (1);
	\draw [style=arrow] (5) to (1);
	\draw [style=arrow] (4) to (1);
	\draw [style=arrow] (4) to (2);
	\draw [style=arrow] (5) to (2);
	\draw [style=arrow] (6) to (2);
	\draw [style=arrow] (3) to (2);
	\draw [style=arrow] (12) to (20);
	\draw [style=arrow] (18) to (20);
	\draw [style=arrow] (13) to (20);
	\draw [style=arrow] (13) to (15);
	\draw [style=arrow] (18) to (15);
	\draw [style=arrow] (12) to (15);
	\draw [style=arrow] (19) to (15);
	\draw [style=arrow] (3) to (20);
	\draw [style=arrow] (29) to (35);
	\draw [style=arrow] (25) to (35);
	\draw [style=arrow] (33) to (35);
	\draw [style=arrow] (27) to (35);
	\draw [style=arrow] (27) to (28);
	\draw [style=arrow] (33) to (28);
	\draw [style=arrow] (25) to (28);
	\draw [style=arrow] (32) to (28);
	\draw [style=arrow] (41) to (47);
	\draw [style=arrow] (37) to (47);
	\draw [style=arrow] (45) to (47);
	\draw [style=arrow] (39) to (47);
	\draw [style=arrow] (39) to (40);
	\draw [style=arrow] (45) to (40);
	\draw [style=arrow] (37) to (40);
	\draw [style=arrow] (44) to (40);
	\draw [style=arrow, bend left=15, looseness=1.00] (0) to (49);
	\draw [style=arrow] (3) to (49);
	\draw [style=arrow] (19) to (49);
	\draw [style=arrow] (29) to (49);
	\draw [style=arrow] (32) to (49);
	\draw [style=arrow] (41) to (49);
	\draw [style=arrow, bend right=15, looseness=1.00] (44) to (49);
	\end{pgfonlayer}
	\end{tikzpicture}
	\caption{The digraph $F$.}
	\label{F.}
\end{center}
\end{figure}
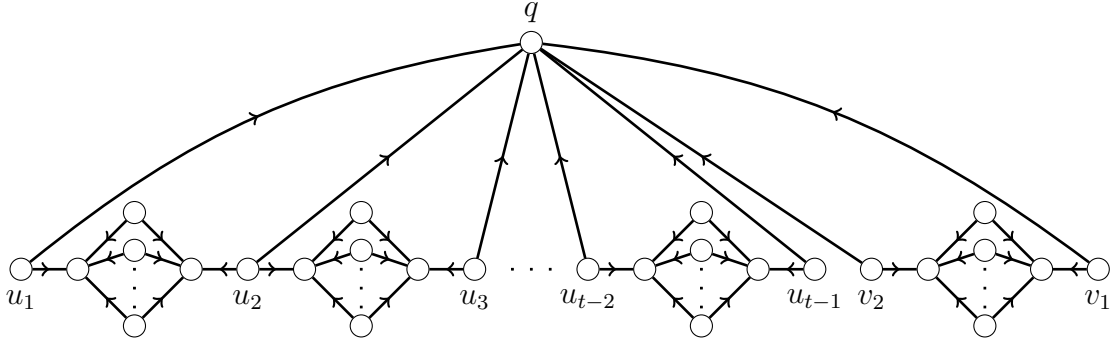

To construct $^*G$ from $G$, replace every edge $uv$ of $G$ with a copy of $F$, such that $u=u_1$ and $v=v_1$. The construction can clearly be accomplished in polynomial time. Note that while the edge gadget is not symmetric, it does not matter in what follows. We now show that $^*G$ has a $t$-frugal $H$-colouring if and only if $G$ has an $H^*$ colouring.

First, suppose that $^*G$ has a $t$-frugal $H$-colouring $^*c$. We claim that the colouring $c$ obtained by setting $c(v) =$ $^*c(v)$ is a proper $H^*$-colouring of $G$. To do this, we must show that for every $uv \in E(G)$, $c(u)c(v) \in E(H^*)$. By the definition of $H^*$, we must show that $c(u)$ and $c(v)$ are distinct, that $c(u)$ and $c(v)$ have a common out-neighbour in $H$, and that each of $c(u)$ and $c(v)$ have an out-neighbour in $H$ with in-degree $\Delta^-(H)$. Recall that in any $t$-frugal $H$-colouring of $F$, the vertices $u=u_1,u_2,\ldots,u_{t-1}$ must all receive the same colour and the vertices $v=v_1$ and $v_2$ must receive the same colour. Therefore, if $c(u) = c(v)$ the vertex $q$ has $t+1$ in-neighbours of the same colour, which violates the frugality condition, and hence $c(u)$ and $c(v)$ must be distinct. The vertices $u$ and $v$ in each copy of $F$ have a common out-neighbour in $F$, namely $q$, and therefore $c(u)$ and $c(v)$ have $c(q)$ as a common out-neighbour in $H$. Finally, the vertices $u$ and $v$ in each copy of $F$ have out-neighbours with in-degree $t\Delta^-(H)$, which can only be coloured by a vertex of $H$ with in-degree $\Delta^-(H)$, and hence $c(u)$ and $c(v)$ must each have an out-neighbour of in-degree $\Delta^-(H)$ in $H$. Therefore, $c$ is a proper $H^*$-colouring of $G$.

Now, suppose that $G$ has an $H^*$-colouring $c$. First, notice that the vertices of $^*G$ which correspond to vertices of $G$ have only out-neighbours, and therefore the frugality condition is vacuously satisfied for these vertices. Therefore, if we can extend a partial colouring of $F$ where only the vertices $u=u_1$ and $v=v_1$ are coloured to a $t$-frugal $H$-colouring of $F$, we can simply do this to each copy of $F$ in $^*G$ and we will obtain a $t$-frugal $H$-colouring of $^*G$. Consider an edge $uv$ of $G$, and let $a$ and $b$ denote $c(u)$ and $c(v)$ respectively. Recall that $F$ is composed of a copy of $F_1$, a copy of $F_0$, and a vertex $q$. Let $^*c(u_i) = a$ for $1 \leq i \leq t-1$, and let $^*c(v_1)=$ $^*c(v_2)=b$. The vertices with in-degree $t\Delta^-(H)$ will then have a vertex coloured $a$ or $b$ as an in-neighbour. 
By the definition of $E(H^*)$, the vertices $a$ and $b$ have out-neighbours $d$ and $e$, respectively, with in-degree $\Delta^-(H)$ in $H$. 
Colour the vertices of in-degree $t\Delta^-(H)$ that have an in-neighbour coloured $a$ with $d$, and those that have an in-neighbour coloured $b$ with $e$. We now have the groups of $t\Delta^-(H) - 1$ vertices, which have common out-neighbours that are either both coloured $d$, or both coloured $e$. Either way, we can colour these vertices greedily with the in-neighbours of $d$ or $e$ respectively while respecting the frugality condition, as $d$ and $e$ have in-degree $\Delta^-(H)$ in $H$. Furthermore, since $a$ and $b$ have a common out-neighbour $f$ in $H$,  we can use the colour $f$ to colour $q$. This colours every vertex of $F$. It remains to check that this colouring satisfies the frugality condition. The only vertices with in-neighbours in $F$ are the vertices with in-degree $t\Delta^-(H)$ (for which we have already verified the frugality condition) and the vertex $q$. Since $a$ and $b$ are distinct vertices of $H$, the frugality condition is satisfied for $q$. The frugality condition is vacuously satisfied for vertices with no in-neighbours. This completes the proof.
\end{proof}


Recall that in Corollary~\ref{linkSAT}, we showed that monotone $\ell$-in-$k$-SAT is NP-complete whenever $k \geq 3$. We can use this to address the case where $H^*$ is a bipartite graph, and obtain the following result.

\begin{theo}
Let $H$ be a directed graph with $\Delta^-(H) \geq 2$, and let $t \geq 2$. Then, $t$-frugal $H$-colouring is NP-complete.
\end{theo}

\begin{proof}
If $H^*$ is not bipartite, then $H^*$-colouring is NP-complete \cite{HN} and by Theorem~\ref{reduc} $t$-frugal $H$-colouring is NP-complete. Therefore, we may assume that $H^*$ is bipartite. Notice that this implies $\Delta^-(H) = 2$, as the set of in-neighbours of a vertex of $H$ with in-degree $\Delta^-(H)$ induces a copy of $K_{\Delta^-(H)}$ in $H^*$.

The transformation is from monotone $t$-in-$2t$-SAT. Let an instance $S$ of monotone $t$-in-$2t$-SAT
be given. Construct a digraph $G$ as follows. For every clause $c$ of $S$, create a vertex $v_c$, and for every variable $\ell$ of $S$, create a vertex $u_{\ell}$. Then, whenever $\ell$ is a variable of the clause $c$, add an arc in $G$ from $u_{\ell}$ to $v_c$. This completes the construction of $G$. This construction can clearly be completed in polynomial time.

Consider a $t$-frugal $H$-colouring, $f$, of $G$. Let $h^*$ be a proper $2$-colouring of $H^*$, using colours $0$ and $1$, and let $h$ be the $2$-colouring of $H$ induced by $h^*$. We construct a truth assignment for $S$ by setting the truth value of the variable $\ell$ to be the value of the colour $h \circ f(u_{\ell})$. Consider a clause $c$ of $S$, and let $f(v_c) = x$. Note that $v_c$ has in-degree $2t$, and therefore $x$ must be a vertex of in-degree $2$ in $H$. Let $y$ and $z$ denote the in-neighbours of $x$. Since $f$ is a $t$-frugal colouring, exactly $t$ of the $2t$ in-neighbours of $v_c$ must be coloured $y$, and exactly $t$ of them must be coloured $z$. Then, since $y$ and $z$ have a common out-neighbour, $x$, in $H$, they are adjacent in $H^*$. Therefore $h(y) \neq h(z)$, and the constructed truth assignment for $S$ is a valid truth assignment where every clause has exactly $t$ true variables.

Now, consider a satisfying truth assignment for $S$ where each clause has exactly $t$ true variables. 
We construct a $t$-frugal $H$-colouring, $f$, of $G$ as follows.
Let $x$ be a vertex of $H$ with in-degree $2$ and in-neighbours $y$ and $z$. For each clause $c$, let $f(v_c) = x$, and for each variable $\ell$, let $f(u_{\ell}) = y$ if $\ell$ is assigned the value true, and let $f(u_{\ell}) =z $ if $\ell$ is assigned the value false. The resulting colouring $f$ is clearly $t$-frugal because each clause contains exactly $t$ variables assigned the value true, and $t$ variables assigned the value false. This completes the proof.
\end{proof}

We are left to consider the case where $\Delta^-(H) = 1$. A subgraph $H'$ of $H$ is a \emph{retract} of $H$ if there exists a homomorphism $h$ from $H$ to $H'$ such that $h(v) = v$ for every vertex $v$ of $H'$. A graph $H$ with no proper retract is known as a \emph{core} \cite{HNBook}.

\begin{lemma}
\label{core}
Let $H$ be a directed graph with $\Delta^-(H) = 1$. Let $H'$ be a retract of $H$. Then a graph $G$ has a $t$-frugal $H$-colouring if and only if $G$ has a $t$-frugal $H'$-colouring.
\end{lemma}

\begin{proof}
A $t$-frugal $H'$-colouring of $G$ induces a $t$-frugal $H$-colouring of $G$.

Let $f$ be a $t$-frugal $H$-colouring of $G$, and let $h$ be a retraction from $H$ to $H'$. We claim that $h \circ f$ is a $t$-frugal $H'$-colouring of $G$. Since the composition of two homomorphisms is a homomorphism, we need only show that $h \circ f$ is $t$-frugal. We know that $\Delta^-(H) = 1$, and therefore the homomorphism $f$ maps all in-neighbours of any vertex $v$ of $G$ to the same vertex of $H$. Hence, $G$ has maximum in-degree $t$ and therefore $h \circ f$ is $t$-frugal.
\end{proof}

\begin{theo}
If $H$ is a directed graph with $\Delta^-(H) = 1$, then $t$-frugal $H$-colouring is solvable in polynomial time.
\end{theo}

\begin{proof}
By Lemma~\ref{core} we may assume $H$ is a core. The core of a directed graph $H$ with $\Delta^-(H) = 1$ is either a directed path, a disjoint union of directed cycles which have lengths that do not divide one another, or a loop. We proceed by cases.

First note that when $H$ is a loop, a given digraph $G$ has a $t$-frugal homomorphism to $H$ if and only if $G$ has maximum in-degree $t$, which can be easily determined.

Now suppose $H$ is a directed path.
Consider an instance $G$ of $t$-frugal $H$-colouring. If $G$ has maximum in-degree greater than $t$, then a $t$-frugal $H$-colouring of $G$ does not exist, and the maximum in-degree of $G$ can easily be determined. 
Otherwise, $G$ has maximum in-degree at most $t$, and any $H$-colouring of $G$ is $t$-frugal. 
The existence of a homomorphism to a directed path
can be checked in polynomial time \cite{MSW}.


Finally suppose $H$ is a disjoint union of directed cycles $C_1,C_2,\ldots,C_k$, of pairwise relatively prime lengths.
Consider an instance $G$ of $t$-frugal $H$-colouring. 
As before, if $G$ has maximum in-degree greater than $t$ then a $t$-frugal $H$-colouring of $G$ does not exist, and if $G$ has maximum in-degree at most $t$ then any $H$-colouring of $G$ is $t$-frugal. 
The existence of a homomorphism to a directed cycle, and therefore to a vertex-disjoint union of directed cycles,
can be checked in polynomial time \cite{MSW}.


This completes the proof.
\end{proof}

Combining our results, we obtain a dichotomy theorem for $t$-frugal $H$-colouring.

\begin{theo}
Let $H$ be a directed graph, and let $t \geq 2$. Then, the problem $t$-frugal $H$-colouring is solvable in polynomial time if $\Delta^-(H) = 1$, and NP-complete if $\Delta^-(H) \geq 2$.
\end{theo}

	\bibliographystyle{plain}
	\bibliography{References}{}

@incollection {Schaefer,
	AUTHOR = {Schaefer, Thomas J.},
	TITLE = {The complexity of satisfiability problems},
	BOOKTITLE = {Conference {R}ecord of the {T}enth {A}nnual {ACM} {S}ymposium
	on {T}heory of {C}omputing ({S}an {D}iego, {C}alif., 1978)},
	PAGES = {216--226},
	PUBLISHER = {ACM, New York},
	YEAR = {1978},
	MRCLASS = {68C25 (03B05 03D15 05C15 68E10)},
	MRNUMBER = {521057},
	MRREVIEWER = {M. I. Dekhtyar},
}

@article {BBDMY,
	AUTHOR = {Bard, Stefan and Bellitto, Thomas and Duffy, Christopher and
	MacGillivray, Gary and Yang, Feiran},
	TITLE = {Complexity of locally-injective homomorphisms to tournaments},
	JOURNAL = {Discrete Math. Theor. Comput. Sci.},
	FJOURNAL = {Discrete Mathematics \& Theoretical Computer Science. DMTCS.},
	VOLUME = {20},
	YEAR = {2018},
	NUMBER = {2},
	PAGES = {Paper No. 4, 22},
	MRCLASS = {05C20 (68Q17 68Q25 68R10)},
	MRNUMBER = {3887371},
}

@article{HN,
	title = "On the complexity of {$H$}-coloring",
	journal = "Journal of Combinatorial Theory, Series B",
	volume = "48",
	number = "1",
	pages = "92 - 110",
	year = "1990",
	issn = "0095-8956",
	doi = "https://doi.org/10.1016/0095-8956(90)90132-J",
	url = "http://www.sciencedirect.com/science/article/pii/009589569090132J",
	author = "Pavol Hell and Jaroslav Nešetřil",
}

@INPROCEEDINGS{B17,
	author = {Bulatov, Andrei},
	booktitle = {2017 IEEE 58th Annual Symposium on Foundations of Computer Science (FOCS)},
	title = {A Dichotomy Theorem for Nonuniform {CSP}s},
	year = {2017},
	pages = {319-330},
	doi = {10.1109/FOCS.2017.37},
	url = {doi.ieeecomputersociety.org/10.1109/FOCS.2017.37},
	ISSN = {0272-5428},
}

@incollection {Z17,
	AUTHOR = {Zhuk, Dmitry},
	TITLE = {An algorithm for constraint satisfaction problem},
	BOOKTITLE = {2017 {IEEE} 47th {I}nternational {S}ymposium on
	{M}ultiple-{V}alued {L}ogic---{ISMVL} 2017},
	PAGES = {1--6},
	PUBLISHER = {IEEE, New York},
	YEAR = {2017},
	MRCLASS = {68Q17 (08A70)},
	MRNUMBER = {3703720},
}

@article{inj8,
	author    = {Gary MacGillivray and
	Jacobus Swarts},
	title     = {The complexity of locally injective homomorphisms},
	journal   = {Discrete Mathematics},
	volume    = {310},
	number    = {20},
	pages     = {2685--2696},
	year      = {2010},
	url       = {http://dx.doi.org/10.1016/j.disc.2010.03.034},
	doi       = {10.1016/j.disc.2010.03.034},
	bibsource = {dblp computer science bibliography, http://dblp.org}
}

@article{Russell,
	author = {Russell J. Campbell},
	title = {Reflexive Injective Oriented Colourings},
	journal = {M.Sc. Thesis, Department of Mathematics and Statistics, University of Victoria, Canada},
	year = {2009}
}

@article {Rus1,
	AUTHOR = {Campbell, Russell J. and Clarke, Nancy E. and MacGillivray, Gary},
	TITLE = {Injective Homomorphisms to Small Tournaments},
	JOURNAL = {Open Journal of Discrete Mathematics},
	YEAR = {2023},
	PAGES = {1--15},
	VOLUME = {13}
}

@inproceedings{FederVardi,
	author = {Feder, Tom\'{a}s and Vardi, Moshe Y.},
	title = {Monotone Monadic {SNP} and Constraint Satisfaction},
	year = {1993},
	isbn = {0897915917},
	publisher = {Association for Computing Machinery},
	address = {New York, NY, USA},
	url = {https://doi.org/10.1145/167088.167245},
	doi = {10.1145/167088.167245},
	booktitle = {Proceedings of the Twenty-Fifth Annual ACM Symposium on Theory of Computing},
	pages = {612–622},
	numpages = {11},
	location = {San Diego, California, USA},
	series = {STOC '93}
}

@article{SBard,
	author = {Stefan Bard},
	title = {Complexity of Frugal Homomorphisms},
	journal = {Ph.D.  Thesis, Department of Mathematics and Statistics, University of Victoria, Canada},
	year = {2021}
}

@book {HNBook,
    AUTHOR = {Hell, Pavol and Ne\u{s}et\u{r}il, Jaroslav},
     TITLE = {Graphs and homomorphisms},
    SERIES = {Oxford Lecture Series in Mathematics and its Applications},
    VOLUME = {28},
 PUBLISHER = {Oxford University Press, Oxford},
      YEAR = {2004},
     PAGES = {xii+244},
      ISBN = {0-19-852817-5},
   MRCLASS = {05-02 (05C15 05C60 18B15 68R10)},
  MRNUMBER = {2089014},
MRREVIEWER = {Richard\ C.\ Brewster},
       DOI = {10.1093/acprof:oso/9780198528173.001.0001},
       URL = {https://doi.org/10.1093/acprof:oso/9780198528173.001.0001},
}

@article {MSW,
    AUTHOR = {Maurer, H. A. and Sudborough, J. H. and Welzl, E.},
     TITLE = {On the complexity of the general coloring problem},
   JOURNAL = {Inform. and Control},
  FJOURNAL = {Information and Control},
    VOLUME = {51},
      YEAR = {1981},
    NUMBER = {2},
     PAGES = {128--145},
      ISSN = {0019-9958},
   MRCLASS = {68C25 (03D15 05C15 68E10)},
  MRNUMBER = {686834},
       DOI = {10.1016/S0019-9958(81)90226-6},
       URL = {https://doi.org/10.1016/S0019-9958(81)90226-6},
}

\end{document}